\documentclass{amsart}

\usepackage{float}
\usepackage{amsthm}
\usepackage{amsmath}
\usepackage{amsfonts}
\usepackage{amssymb}
\usepackage{latexsym}
\usepackage{mathrsfs}
\usepackage{geometry}
\usepackage{graphicx}
\usepackage{color, soul}
\usepackage{accents}

\newcommand{\sca}[1]{\left\langle#1\right\rangle} %\sca{x,y}
\newcommand{\abs}[1]{\lvert#1\rvert}

\newtheorem{theorem}{Theorem}[section]
\newtheorem{lemma}[theorem]{Lemma}
\newtheorem{proposition}[theorem]{Proposition}
\newtheorem{corollary}[theorem]{Corollary}

\theoremstyle{definition}
\newtheorem{definition}[theorem]{Definition}

\theoremstyle{remark}

\newtheorem*{acknowledgment}{Acknowledgments}

\DeclareMathOperator{\E}{\mathbb{E}} % expectation symbol
\numberwithin{equation}{section}

\begin{document}

\title{Random matrices and controllability of dynamical systems}

%   Information for first author
\author{John Leventides}
\address{Department of Economics, Faculty of Economics and Political Sciences, National and Kapodistrian University of Athens.}
\email{ylevent@econ.uoa.gr}
 %   \thanks will become a 1st page footnote.

%   Information for first author
\author{Nick Poulios}
\address{Department of Economics, Faculty of Economics and Political Sciences, National and Kapodistrian University of Athens.}
\email{npoulios@econ.uoa.gr}
 %   \thanks will become a 1st page footnote.

%    Information for third author
\author{Costas Poulios}
\address{Department of Economics, Faculty of Economics and Political Sciences, National and Kapodistrian University of Athens.}
\email{konpou@econ.uoa.gr}

\keywords{Controllability of dynamical systems, random matrices, random systems.\\
Institute of Mathematics and Its Applications (IMA) U.K., 2020, accepted}

\date{}
%\dedicatory{}
%\commby{}

\begin{abstract}
We introduce the concept of $\epsilon$-uncontrollability for random linear systems, i.e. linear system in which the usual matrices have been replaced by random matrices. We also estimate the $\epsilon$-uncontrollability in the case where the matrices come from the Gaussian orthogonal ensemble. Our proof utilizes tools from systems theory, probability theory and convex geometry.
\end{abstract}

\maketitle

\section{Introduction}\label{sec.introduction}
Controllability is one of the most fundamental concepts in systems theory and control theory. Roughly speaking, a system is controllable if one can switch from one trajectory to the other provided that the laws governing the system are obeyed and some delay is allowed. In the present work, we focus on the linear, time-invariant multivariable system described by the following equations (see \cite{kark}, \cite{polder}):
\begin{equation}\label{eq.system}
  \frac{dx}{dt} = \mathbf{Ax}+ \mathbf{bu}, \quad \mathbf{y}=\mathbf{Cx}+\mathbf{du},
\end{equation}
where $\mathbf{x}=x(t)=(x_1(t), x_2(t), \ldots, x_n(t))^t\in \mathbb{R}^n$ is a vector describing the state of the system at time $t$, $\mathbf{u}=u(t)\in\mathbb{R}$ is the input and $\mathbf{y}=y(t)$ is the $m$-vector of outputs. $\mathbf{A}$, $\mathbf{C}$ are respectively $n\times n$, $m\times n$ matrices, and $\mathbf{b}\in \mathbb{R}^n$, $\mathbf{d}\in \mathbb{R}^m$ are vectors. In this case the definition of controllability goes as follows. The system \eqref{eq.system} is said to be \emph{state controllable} or simply \emph{controllable}, if there exists a finite time $T>0$, such that for any initial state $x(0)\in\mathbb{R}^n$ and any $x_1\in\mathbb{R}^n$, there is an input $\mathbf{u}=u(t)$ defined on $[0,T]$ that will transfer $x(0)$ to $x_1$ at time $T$ (i.e. $x(t)$ obeys the first equation of \eqref{eq.system} and $x(T)=x_1$). Otherwise, the system \eqref{eq.system} is called \emph{uncontrollable}.

In this article, we consider random systems of the form \eqref{eq.system}, that are systems where the parameters $\mathbf{A}$ and $\mathbf{b}$ have been replaced with random matrices. Given a positive number $\epsilon$, we define the concept of $\epsilon$-uncontrollability of a random system. It is natural that the $\epsilon$-uncontrollability of a random system depends on the distribution of the entries of $\mathbf{A}$ and $\mathbf{b}$. Consequently, we consider the fundamental Gaussian orthogonal ensemble of random matrices and we calculate the $\epsilon$-uncontrollability in this particular case.

The rest of the paper is organised as follows. In Section \ref{sec.uncontrol}, we define the $\epsilon$-uncontrollability for random systems. In Section \ref{sec.Gaussian}, we describe the Gaussian orthogonal ensemble, which is going to be used in this work, and we state some known results about this ensemble that will be used in the sequel. In Section \ref{sec.n=2}, we consider the case $n=2$, i.e. when the state space is $\mathbb{R}^2$, and we give the detailed calculation of the $\epsilon$-uncontrollability of a random system where the matrix $\mathbf{A}$ comes from the Gaussian orthogonal ensemble. Finally, in Section \ref{sec.general}, we deal with the general case of $\mathbb{R}^n$ for $n>2$. In this case, the situation is more complicated and we provide an upper bound for the $\epsilon$-uncontrollability of a random system.

\section{$\epsilon$-uncontrollability of random systems}\label{sec.uncontrol}
In order to formulate a suitable concept of uncontrollability for random systems, we utilize a characterisation of the controllability of systems of the form \eqref{eq.system}. This is provided in the next theorem (for more details see, for example, \cite{kark} Theorem 2.2).

\begin{theorem}
  The following are equivalent.
  \begin{enumerate}
    \item The system \eqref{eq.system} is controllable.
    \item The matrix $[s\mathbf{I-A}, \mathbf{-b}]$ has full rank (i.e. rank $n$), for every $s\in \mathbb{C}$.
  \end{enumerate}
\end{theorem}

Additionally, we need the next lemma whose proof is based on elementary linear algebra and thus it is omitted.

\begin{lemma}
  Let $\mathbf{A}$ be an $n\times n$-matrix and $\mathbf{b}\in \mathbb{R}^n$. Then, the following are equivalent.
  \begin{enumerate}
    \item The matrix $[s\mathbf{I-A}, \mathbf{-b}]$ has full rank (i.e. rank $n$), for every $s\in \mathbb{C}$.
    \item There is no eigenvector $\mathbf{v}\in \mathbb{R}^n$ of the matrix $\mathbf{A}$ such that $\sca{\mathbf{v, b}}=0$.
  \end{enumerate}
\end{lemma}

Motivated by the above results, we now define the $\epsilon$-uncontrollability for random networks.

\begin{definition}
  Assume that a random system is given:
   \begin{equation}\label{eq.random_system}
     \frac{dx}{dt} = \mathbf{A x} +\mathbf{b u},
   \end{equation}
   where $\mathbf{A}$ is an $n\times n$ random matrix and $\mathbf{b}$ is an $n$-dimensional random vector. Given a positive number $\epsilon$, the $\epsilon$-uncontrollability of the above system is defined to be the probability:
  $$P_\epsilon = \mathbb{P} \left( \abs{\sca{\mathbf{v,b}}}<\epsilon \,\, : \,\, \text{for some eigenvector } \mathbf{v} \text{ of } \mathbf{A} \text{ with } \|\mathbf{v}\|_2=1  \right).$$
\end{definition}

\section{Random matrix ensemble}\label{sec.Gaussian}
It is quite evident that the measure of $\epsilon$-uncontrollability of the random system \eqref{eq.random_system} depends on the distribution of the matrix $\mathbf{A}$ and the vector $\mathbf{b}$. In this article, we consider one important ensemble of real symmetric random matrices, namely the so-called \emph{Gaussian orthogonal ensemble (GOE)}. On account of its applications, GOE is one of the most studied random matrix ensembles. It is placed in the more general framework of Wigner matrices, which are defined as follows. We consider $\xi, \zeta$ real-valued random matrices with zero mean. Let $\mathbf{W}=\left(w_{ij}\right)_{i,j=1}^n$ be a random symmetric matrix. We call $\mathbf{W}$ a Wigner matrix if his entries satisfy the next conditions:
\begin{itemize}
  \item $\{w_{ij} \mid 1\le i\le j\le n\}$ are independent random variables;
  \item $\{w_{ij} \mid 1\le i <  j\le n\}$ are i.i.d. (independent, identically distributed) copies of $\xi$;
  \item $\{w_{ii} \mid i =1,\ldots, n\}$ are i.i.d. copies of $\zeta$.
\end{itemize}
The case of Wigner matrices in which $\xi$ and $\zeta$ are Gaussian with $\E[\xi^2]=1$ and $\E[\zeta^2]=2$ gives the Gaussian orthogonal ensemble. Hence, if the symmetric matrix $\mathbf{W}$ belongs to GOE, then $w_{ii} \sim N(0,2)$ (for all $i=1,\ldots,n$), $w_{ij} \sim N(0,1)$ (for all $1\le i<j\le n$) and the entries on and above the diagonal are independent random variables. (We write GOE$(n)$ when an emphasis on the dimension is necessary. However, in majority of cases the dimension will be clear from the context.)

Additionally, as far as the random vector $\mathbf{b}$ is concerned, we have to choose some ensemble. More specifically, we consider the ensemble $\mathbf{S}_b$ containing all random vectors $\mathbf{b}=(b_1, b_2, \ldots, b_n)$ such that $(b_i)_{i=1}^n$  are independent Gaussian random variables with zero mean and $\E[b_i^2]=1$. Furthermore, we assume that $(b_i)_{1\le i\le n}$ and $(w_{ij})_{1\le i \le j \le n}$ are all independent random variables.

For more information concerning Wigner matrices and the GOE we refer to \cite{ande}. For our purpose, we need a couple of result for the eigenstructure of the GOE, which are stated here without proof. First of all, it is known that a.s., the eigenvalues of a matrix $\mathbf{A}$ from GOE are all distinct (see \cite{ande}, Theorem 2.5.2). Let now $v_1, \ldots, v_n$ denote the eigenvectors corresponding to the (real) eigenvalues of $\mathbf{A}$, with their first non zero entry positive real. Then the following proposition holds (see \cite{ande}, Corollary 2.5.4).

\begin{proposition}\label{prop-distri-eigenvector}
  The collection $(v_1,v_2,\ldots, v_n)$ is independent of the eigenvalues. Each of the eigenvectors $v_1,v_2,\ldots, v_n$ is distributed uniformly on
  $$\mathbf{S}^{n-1}_+ = \{ \mathbf{x}= (x_1,x_2,\ldots, x_n) \mid x_i\in\mathbb{R}, \|\mathbf{x}\|_2=1, x_1>0\}.$$
\end{proposition}

\section{The case $n=2$}\label{sec.n=2}
This section is entirely devoted to the calculation of the $\epsilon$-uncontrollability of a random system of the form \eqref{eq.random_system} when the state space is $\mathbb{R}^2$ and $\mathbf{A, b}$ belong to GOE and $\mathbf{S}_b$ respectively. In order to achieve this goal, we firstly fix a vector $\mathbf{b}\in \mathbb{R}^2$ and we set
$$P_{\epsilon, b} = \mathbb{P}\left( \abs{\sca{\mathbf{v,b}}}<\epsilon \,\, : \,\, \text{for some eigenvector } v \text{ of } A, \|\mathbf{v}\|_2=1 \right).$$
Then, the following result holds.
\begin{theorem}\label{th.uncontrol-n=2}
  Let the random system \eqref{eq.random_system} be given, where the state space is $\mathbb{R}^2$ and $\mathbf{A}$ belongs to GOE($2$). Then, for every non zero vector $\mathbf{b}\in \mathbb{R}^2$, we have
  $$P_{\epsilon, \mathbf{b}} = \left\{
                        \begin{array}{ll}
                          \frac{4}{\pi} \cdot \arcsin \left(\frac{\epsilon}{\|\mathbf{b}\|_2}\right), & \hbox{if $\frac{\epsilon}{\|b\|_2}<\frac{\sqrt{2}}{2}$;} \\
                          1, & \hbox{if $\frac{\epsilon}{\|b\|_2} \ge \frac{\sqrt{2}}{2}$.}
                        \end{array}
                      \right.$$
\end{theorem}

\begin{proof}
Since the matrix $\mathbf{A}$ from GOE($2$) is symmetric, there is an orthonormal basis $\{\mathbf{v_1, v_2}\}$ of $\mathbb{R}^2$ consisting of eigenvectors of $\mathbf{A}$. Without loss of generality (replacing $\mathbf{v}_1$ with $\mathbf{-v}_1$ or changing the order of $\{\mathbf{v_1,v_2}\}$, if necessary), we may assume that the first coordinate of $\mathbf{v}_1$ is positive. Hence, we can write $\mathbf{v}_1 = (\cos \theta,\sin \theta)$ and $\mathbf{v}_2 = (-\sin\theta, \cos\theta)$, for some $\theta \in (-\frac{\pi}{2}, \frac{\pi}{2})$. Now, we have:
 \begin{align*}
   P_{\epsilon, \mathbf{b}} =& \mathbb{P}\left( \abs{\sca{\mathbf{v,b}}}<\epsilon \,\, : \,\, \text{for some eigenvector } \mathbf{v} \text{ of }\mathbf{A} \text{ with } \|v\|_2=1  \right)\\
   = & \mathbb{P}\left( \abs{\sca{\mathbf{v}_i,\mathbf{b}}}<\epsilon \,\, : \,\, \text{for some eigenvector } \mathbf{v}_i \text{ of }\mathbf{A}, i=1,2   \right).
 \end{align*}
The non zero vector $\mathbf{b}$ is written in the form $\mathbf{b}=\|b\|_2 \left( \frac{b_1}{\|\mathbf{b}\|_2}, \frac{b_2}{\|\mathbf{b}\|_2} \right)$. Let $T$ be the rotation through a suitable angle $\varphi$ such that $T  \left( \frac{b_1}{\|\mathbf{b}\|_2}, \frac{b_2}{\|\mathbf{b}\|_2} \right) =e_2 = (0,1) $. Then, $T$ is in orthogonal group $O(2)$ and hence,
$$\abs{\sca{\mathbf{v,b}}} = \abs{\sca{T(\mathbf{v}),T(\mathbf{b})}} = \|\mathbf{b}\|_2 \cdot \abs{\sca{T(\mathbf{v}),e_2}}$$
for any eigenvector $\mathbf{v}=v_i$, $i=1,2$ of $\mathbf{A}$. Thus,
\begin{align*}
  P_{\epsilon, \mathbf{b}} = & \mathbb{P}\left( \|\mathbf{b}\|_2 \cdot \abs{\sca{e_2,T(\mathbf{v}_i)}}<\epsilon \,\, : \,\, \text{for } i=1 \text{ or } i=2  \right) \\
  = & \mathbb{P}\left( \abs{\sca{e_2,T(\mathbf{v}_i)}}<\frac{\epsilon}{\|\mathbf{b}\|_2} \,\, : \,\, \text{for } i=1 \text{ or } i=2  \right).
\end{align*}
The rotation $T$ is given by the following matrix representation
$$T= \left[
\begin{array}{cc}
  \cos \varphi & -\sin\varphi \\
  \sin\varphi  & \cos \varphi
\end{array}
\right].$$
Therefore,
 $$T(\mathbf{v}_1) = (\cos(\varphi+\theta),\sin(\varphi+\theta)) \quad \text{ and } \quad T(\mathbf{v}_2) = (-\sin(\varphi+\theta),\cos(\varphi+\theta)),$$
and, consequently,
$$P_{\epsilon, \mathbf{b}} =  \mathbb{P}\left( \abs{\sin(\varphi+\theta)}<\frac{\epsilon}{\|\mathbf{b}\|_2} \,\, \text{ or } \,\,\abs{\cos(\varphi+\theta)} <\frac{\epsilon}{\|\mathbf{b}\|_2}  \right).$$
We have to distinguish two cases.
\begin{description}
  \item [Case I] If $\frac{\epsilon}{\|\mathbf{b}\|_2} \ge \frac{\sqrt{2}}{2}$, then, clearly, one has $P_{\epsilon, \mathbf{b}}=1$.
  \item [Case II] If $\frac{\epsilon}{\|\mathbf{b}\|_2} < \frac{\sqrt{2}}{2}$, then
  $$ P_{\epsilon, \mathbf{b}} = \mathbb{P}\left( \abs{\sin(\varphi+\theta)}<\frac{\epsilon}{\|\mathbf{b}\|_2} \right) + \mathbb{P}\left( \abs{\cos(\varphi+\theta)}<\frac{\epsilon}{\|\mathbf{b}\|_2}  \right),$$
where the equality follows from the disjointness of the two sets. Now, for the first summand, we observe that $\varphi+\theta$ belongs to a semicircle. Hence, the values of $\varphi+\theta$ for which we have $\abs{\sin(\varphi+\theta)}<\frac{\epsilon}{\|\mathbf{b}\|_2}$ belong to an arc or to the unions of two disjoint arcs whose length is $2\cdot \arcsin(\frac{\epsilon}{\|\mathbf{b}\|_2})$. It follows that $\theta$ belongs either to an arc or to the union of two disjoint arcs with total length $2\cdot \arcsin(\frac{\epsilon}{\|\mathbf{b}\|_2})$. By Proposition \ref{prop-distri-eigenvector}, $\theta$ is a random variable with the uniform distribution on the interval $(-\frac{\pi}{2},\frac{\pi}{2})$. Therefore,
$$\mathbb{P}\left( \abs{\sin(\varphi+\theta)}<\frac{\epsilon}{\|\mathbf{b}\|_2} \mid \theta\in (-\frac{\pi}{2}, \frac{\pi}{2}) \right) = \frac{2}{\pi} \cdot \arcsin(\frac{\epsilon}{\|\mathbf{b}\|_2}).$$
Using similar argumentation for the second summand, we finally obtain:
$$ P_{\epsilon, \mathbf{b}} = \frac{4}{\pi} \cdot \arcsin(\frac{\epsilon}{\|\mathbf{b}\|_2}).$$
\end{description}
\end{proof}

We are now ready to prove the main result of this section.
\begin{theorem}
  \label{th.main_uncontrol-n=2}
Assume that $n=2$ and that $\mathbf{A,b}$ belong to GOE and $\mathbf{S_b}$ respectively. For any positive number $\epsilon$, the $\epsilon$-uncontrollability of the random system \eqref{eq.random_system} is given by
  $$P_{\epsilon} = 1-e^{2\epsilon^2}+ \frac{4}{\sqrt{2}\pi}\int_{2\epsilon^2}^{\infty} \arcsin\left( \frac{\epsilon}{\sqrt{x}}\right) e^{-x/2} \, dx .$$
\end{theorem}

\begin{proof}
Let $\{\mathbf{v}_1, \mathbf{v}_2\}$ be an orthonormal basis of $\mathbb{R}^2$ consisting of eigenvectors of $\mathbf{A}$. We set
$$Z= \min \{\abs{\sca{\mathbf{v}_1,\mathbf{b}}}, \abs{\sca{\mathbf{v}_2,\mathbf{b}}} \}.$$
Then, $Z$ is a non-negative random variable, which follows from the coordinates of the random vectors $\mathbf{v}_1, \mathbf{v}_2, \mathbf{b}$ after multiplication, summation and absolute values. It is not hard to see that
$$P_\epsilon = \mathbb{P} (Z<\epsilon) = \mathbb{E} [Z\cdot \mathbf{1}_{[0,\epsilon]}],$$
where $ \mathbf{1}_{[0,\epsilon]}$ is the characteristic (or indicator) function of the interval.

Recall (from Section \ref{sec.Gaussian}) our assumption that the entries of $\mathbf{A}$ and $\mathbf{b}$ are independent. It follows that $\mathbf{v}_1, \mathbf{b}$ are independent random vectors and clearly this is also true for the pair $\mathbf{v}_2, \mathbf{b}$. Using conditional expectation, we obtain that
$$P_\epsilon = \mathbb{E} [Z\cdot \mathbf{1}_{[0,\epsilon]}] = \int_{\mathbb{R}^2} \mathbb{E} [Z\cdot \mathbf{1}_{[0,\epsilon]} | \mathbf{b}] \cdot f(\mathbf{b}) \, d\mathbf{b} = \int_{\mathbb{R}^2} P_{\epsilon, b} \cdot f(\mathbf{b}) \, d\mathbf{b},$$
where $f(\mathbf{b})$ is the probability density function of the random vector $\mathbf{b}$. Since the coordinates $\mathbf{b}_1,\mathbf{b}_2$ of $\mathbf{b}$ are independent Gaussian random variables with zero mean and variance equal to $1$, it follows that
$$
  P_\epsilon =    \int_{\mathbb{R}^2} P_{\epsilon, \mathbf{b}} \cdot  \frac{1}{2\pi} \cdot  \exp\left(- \frac{\mathbf{b}_1^2+\mathbf{b}_2^2}{2}\right) \, d\mathbf{b}_1d\mathbf{b}_2 =    \int_{\mathbb{R}^2} P_{\epsilon, b} \cdot  \frac{1}{2\pi} \cdot  \exp\left(- \frac{\|\mathbf{b}\|_2^2}{2}\right) \, d\mathbf{b}_1d\mathbf{b}_2.
$$
We observe now that $P_{\epsilon,\mathbf{b}}$ depends only on $\|\mathbf{b}\|_2=\sqrt{\mathbf{b}_1^2+\mathbf{b}_2^2}$. Hence, by changing in polar coordinates (or, equivalently using the fact that $\mathbf{b}_1^2+\mathbf{b}_2^2$ has the $\chi^2$-distribution with $2$ degrees of freedom), we get that
\begin{align*}
  P_\epsilon =& \int_{0}^{2\pi}\int_{0}^{\infty} P_{\epsilon, \mathbf{b}}(r) \frac{1}{2\pi} e^{-r^2/2} r \, drd\theta = \int_{0}^{\infty} P_{\epsilon, \mathbf{b}}(r) e^{-r^2/2} r \, dr =\int_{0}^{\infty} P_{\epsilon, \mathbf{b}}(\sqrt{r}) \frac{1}{2} e^{-r/2}  \, dr \\
  = & \int_{0}^{2\epsilon^2} \frac{1}{2} e^{-r/2}  \, dr + \int_{2\epsilon^2}^{\infty}\frac{4}{\pi} \cdot \arcsin(\frac{\epsilon}{\sqrt{r}}) \frac{1}{2} e^{-r/2}  \, dr\\
 = & 1-e^{2\epsilon^2} + \frac{2}{\pi}\int_{2\epsilon^2}^{\infty} \arcsin(\frac{\epsilon}{\sqrt{r}}) \cdot e^{-r/2}  \, dr
\end{align*}
and we have proved the desired result.
\end{proof}

\section{The general case}\label{sec.general}
In this section, we consider the more general case where the state space of the random system \eqref{eq.random_system} is $\mathbb{R}^n$, $n\ge3$. This case is more complicated and we only give an upper bound for the $\epsilon$-uncontrollability of the system.

Firstly, we need some estimates from the elementary convex geometry. Assume that $\mathbf{v}\in \mathbf{S}^{n-1}$ is a unit vector and $\epsilon\in[0,1)$. The $\epsilon$-spherical cap about $\mathbf{v}$ is the following subset of $\mathbf{S}^{n-1}$:
$$C(\epsilon, \mathbf{v}) =\left\{\theta \in \mathbf{S}^{n-1}  \,\, :\,\, \sca{\theta, \mathbf{v}} \ge \epsilon\right\}.$$
Observe that the number $\epsilon$ does not refer to the radius of the cap. An easy calculation shows that the radius is $r=2(1-\epsilon)$. In general, the cap of radius $r$ about $\mathbf{v}$ is:
$$ \left\{\theta \in \mathbf{S}^{n-1} \,\, :\,\, \abs{\theta-\mathbf{v}} \le r\right\}.$$
Let $A_n$ denote the surface area of the unit ball $S^{n-1}$, i.e. $A_n=\frac{2\pi^{n/2}}{\Gamma(n/2)}$. Convex geometry provides the following upper and lower bounds for the surface area of a spherical cap (see, for example, \cite{ball}).

\begin{lemma}\label{lemma.estimate-1}
  For $0 \le \epsilon <1$, the cap $C(\epsilon, v) $ on $S^{n-1}$ has surface area at most $e^{-n\epsilon^2/2}\cdot A_n$.
\end{lemma}

\begin{lemma}\label{lemma.estimate-2}
  For $0\le r\le 2$, a cap of radius $r$ on $S^{n-1}$ has surface area at least $\frac{1}{2} \cdot \left(\frac{r}{2}\right)^{n-1}\cdot A_n$.
\end{lemma}

Following the lines of Theorem \ref{th.uncontrol-n=2}, we now prove the next result. Assume that we have a random system of the form \eqref{eq.random_system}, where the matrix $\mathbf{A}$ belongs to GOE($n$).

\begin{theorem}\label{th.general-n}
  Let $\mathbf{A}$ be in the GOE($n$) and let $\mathbf{b}\in \mathbb{R}^n$ be any non zero vector. Then, for the random system \eqref{eq.random_system}, we have the estimate:
  $$P_{\epsilon, \mathbf{b}} = \mathbb{P}\left( \abs{\sca{\mathbf{v,b}}}<\epsilon \,\, : \,\, \text{for some eigenvector } \mathbf{v} \text{ of } \mathbf{A}, \|\mathbf{v}\|_2=1 \right) \le n \Big(1 - \left(1-\frac{\epsilon}{\|\mathbf{b}\|_2}\right)^{n-1} \Big) .$$
\end{theorem}

\begin{proof}
  Let $\{\mathbf{v}_i\}_{i=1}^n$ be an orthonormal basis of $\mathbb{R}^n$ consisting of eigenvectors of the matrix $\mathbf{A}$. Without loss of generality (replacing $\mathbf{v}_i$ with $-\mathbf{v}_i$ if necessary), we may assume that the first non zero coordinate of each $\mathbf{v}_i$ is positive. We now obtain:
  $$P_{\epsilon, \mathbf{b}} = \mathbb{P}\left( \abs{\sca{\mathbf{v}_i,\mathbf{b}}}<\epsilon \,\, : \,\, \text{for some } i=1,2,\ldots,n \right).$$
  We write $\mathbf{b}=\|\mathbf{b}\|_2 \left( \frac{\mathbf{b}_1}{\|\mathbf{b}\|_2} , \frac{\mathbf{b}_2}{\|\mathbf{b}\|_2}, \ldots , \frac{\mathbf{b}_n}{\|\mathbf{b}\|_2} \right)$ and we consider an orthogonal transformation $T\in O(n)$ that assigns $\left( \frac{\mathbf{b}_1}{\|\mathbf{b}\|_2} , \frac{\mathbf{b}_2}{\|\mathbf{b}\|_2}, \ldots , \frac{\mathbf{b}_n}{\|\mathbf{b}\|_2} \right)$ to the vector $\mathbf{e}_1=(1,0,\ldots,0)$. Since $T$ is orthogonal, it follows that:
  $$\abs{\sca{\mathbf{b,v}}} = \abs{\sca{T(\mathbf{b}),T(\mathbf{v})}} = \|\mathbf{b}\|_2 \cdot \abs{\sca{\mathbf{e}_1,T(\mathbf{v})}}$$
  for any eigenvector $\mathbf{v}=v_i$, $i=1,2, \ldots, n$ of $A$. Hence,
  \begin{align*}
    & \mathbb{P}\left( \|\mathbf{b}\|_2 \abs{\sca{T(\mathbf{v}_i),\mathbf{e}_1}}<\epsilon \,\, : \,\, \text{for some } i=1,2,\ldots,n \right) \\
    =  & \mathbb{P}\left(  \abs{\sca{T(\mathbf{v}_i),\mathbf{e}_1}}<\frac{\epsilon}{\|\mathbf{b}\|_2} \,\, : \,\, \text{for some } i=1,2,\ldots,n \right).
  \end{align*}
Note that in $\mathbb{R}^n$ for $n\ge 3$, the sets $( \abs{\sca{T(\mathbf{v}_i),\mathbf{e}_1}}<\frac{\epsilon}{\|\mathbf{b}\|_2})$, $i=1,2,\ldots,n$ are not pairwise disjoint, even for small values of $\epsilon$. Therefore, we cannot repeat the argumentation of the case $n=2$. However, we may proceed as follows
$$
  P_{\epsilon, \mathbf{b}} \le  \sum_{i=1}^{n} \mathbb{P}\left(  \abs{\sca{T(\mathbf{v}_i),\mathbf{e}_1}}<\frac{\epsilon}{\|\mathbf{b}\|_2} \right).
$$
Since $\mathbf{v}_i$ is uniformly distributed in $S^{n-1}_+$ (see Proposition \ref{prop-distri-eigenvector}), we have that $T(\mathbf{v}_i)$ is uniformly distributed to some hemisphere. Therefore, if $\mathbf{A}$ denotes the surface area measure in the sphere $S^{n-1}$, then,
  \begin{align*}
    \mathbb{P}\left( \abs{\sca{T(\mathbf{v}_i),\mathbf{e}_1}}<\frac{\epsilon}{\|\mathbf{b}\|_2} \right) \le & \frac{\mathbf{A}(\theta\in S^{n-1} \,\, : \,\, 0<\theta_1<\frac{\epsilon}{\|\mathbf{b}\|_2})}{A_n/2} \\
    = & \frac{A_n/2 - \mathbf{A}(\theta\in S^{n-1} \,\, : \,\, \frac{\epsilon}{\|\mathbf{b}\|_2} \le\theta_1)}{A_n/2}.
  \end{align*}
The set $\{\theta\in S^{n-1} \,\, : \,\, \frac{\epsilon}{\|\mathbf{b}\|_2} \le\theta_1\}$ is a spherical cap of radius $r=2\left(1-\frac{\epsilon}{\|\mathbf{b}\|_2} \right)$. Hence, by Lemma \ref{lemma.estimate-2}, its surface area is at least $\frac{1}{2}\left(\frac{r}{2}\right)^{n-1}A_n$. Therefore,
$$P_{\epsilon, \mathbf{b}} \le \frac{A_n/2 - \frac{1}{2} \left(\frac{r}{2}\right)^{n-1} A_n}{A_n/2} =1 - \left(\frac{r}{2}\right)^{n-1}  = 1 - \left(1-\frac{\epsilon}{\|\mathbf{b}\|_2}\right)^{n-1}.$$
Hence,
  $$P_{\epsilon, \mathbf{b}} \le \sum_{i=1}^{n} \Big(1 - \left(1-\frac{\epsilon}{\|\mathbf{b}\|_2}\right)^{n-1} \Big) = n \Big(1 - \left(1-\frac{\epsilon}{\|\mathbf{b}\|_2}\right)^{n-1} \Big).$$

\end{proof}

\begin{theorem}
  Assume that $\mathbf{A,b}$ belong to GOE($n$) and $S_\mathbf{b}$ respectively and let $\epsilon$ be any positive number. For the $\epsilon$-uncontrollability of the random system \eqref{eq.random_system}, the following inequality holds
  $$ P_\epsilon \le \frac{n}{2^{n/2}\Gamma(n/2)} \int_{0}^\infty \Big(1 - \left(1-\frac{\epsilon}{\sqrt{r}}\right)^{n-1} \Big) \cdot e^{-r/2} \cdot r^{(n/2)-1} \, dr. $$
  \label{theorem5_4}
\end{theorem}

\begin{proof}
  As in the proof of Theorem \ref{th.main_uncontrol-n=2}, it follows that
$$P_\epsilon = \int_{\mathbb{R}^n} P_{\epsilon, \mathbf{b}} \cdot f(b) \, d\mathbf{b},$$
where $f$ is the probability density function of the random vector $\mathbf{b}$. Since the entries of $\mathbf{b}$ are independent Gaussian random variables with zero mean and variance equal to $1$, we have
$$P_\epsilon = \int_{\mathbb{R}^n} P_{\epsilon, \mathbf{b}} \cdot\frac{1}{\sqrt{(2\pi)^n}} \cdot \exp\left( - \frac{\|\mathbf{b}\|_2^2}{2} \right) \, d\mathbf{b}.$$
By Theorem \ref{th.general-n}, we obtain
$$P_\epsilon \le \int_{\mathbb{R}^n} n \Big(1 - \left(1-\frac{\epsilon}{\|\mathbf{b}\|_2}\right)^{n-1} \Big) \cdot\frac{1}{\sqrt{(2\pi)^n}} \cdot \exp\left( - \frac{\|\mathbf{b}\|_2^2}{2} \right) \, d\mathbf{b}.$$

We observe that, in the last integral, only the norm $\|\mathbf{b}\|_2$ of the vector $\mathbf{b}$ appears. Hence, using polar coordinates (see, for example, \cite{folland} Corollary 2.51), or equivalently, the fact that $\mathbf{b}_1^2+\ldots +\mathbf{b}_n^2$ has the $\chi^2$-distribution with $n$ degrees of freedom, we obtain
\begin{align*}
  P_\epsilon \le & \frac{2\sqrt{\pi^n}}{\Gamma(n/2)} \int_{0}^\infty n \Big(1 - \left(1-\frac{\epsilon}{r}\right)^{n-1} \Big) \cdot\frac{1}{\sqrt{(2\pi)^n}} \cdot \exp\left( - \frac{r^2}{2} \right) r^{n-1} \, dr \\
  = & \frac{n}{2^{n/2}\Gamma(n/2)} \int_{0}^\infty \Big(1 - \left(1-\frac{\epsilon}{\sqrt{r}}\right)^{n-1} \Big) \cdot e^{-r/2} \cdot r^{(n/2)-1} \, dr\\
  or \,=& \frac{n 2^{1-n/2}}{\Gamma(n/2)} \int_{0}^{\infty}\Big(1-\left(1-\frac{\epsilon}{r}\right)^{n-1}\Big)\cdot e^{-r^2/2}\cdot r^{n-1} \, dr,
\end{align*}
and the proof is complete.
\end{proof}

The next corollary shows that the growth of $P_\epsilon$ is at most polynomial of degree $n-1$ with respect to $\epsilon$.

\begin{corollary}\label{cor.main}
  For any integer $n\ge 2$ and any positive number $\epsilon$, we have
  $$P_\epsilon \le \sum_{k=1}^{n-1} (-1)^{k+1} \binom{n-1}{k} \frac{n\Gamma(\frac{n-k}{2})}{2^{k/2}\Gamma(\frac{n}{2})} \epsilon^k.$$
\end{corollary}

\begin{proof}
  Using the binomial expansion formula, we obtain
  $$1 - \left(1-\frac{\epsilon}{\sqrt{r}}\right)^{n-1} = \sum_{k=1}^{n-1} \binom{n-1}{k} (-1)^{k+1} \frac{\epsilon^k}{r^{k/2}}.$$
  By Theorem \ref{theorem5_4}, it follows that
  \begin{align*}
    P_\epsilon \le  & \sum_{k=1}^{n-1} (-1)^{k+1} \binom{n-1}{k} \frac{n}{2^{n/2}\Gamma(\frac{n}{2})} \cdot \epsilon^k \int_{0}^{\infty} e^{-r/2} r^{\frac{n-k}{2}-1} dr \\
    = & \sum_{k=1}^{n-1} (-1)^{k+1} \binom{n-1}{k} \frac{n\Gamma(\frac{n-k}{2})}{2^{k/2}\Gamma(\frac{n}{2})} \epsilon^k
  \end{align*}
  and we have the desired result.
\end{proof}

The next natural corollary is now straightforward.

\begin{corollary}
  For any integer $n\ge 2$, we have that $\lim_{\epsilon \to 0} P_{\epsilon} =0$.
\end{corollary}

Finally, we have the following estimate for the growth rate of $P_\epsilon$ at $0$.

\begin{corollary}
  Assume that $P_\epsilon$ is differentiable at $0$. Then,
  $$ \left. \frac{dP_\epsilon}{d\epsilon} \right|_{\epsilon=0} \le \frac{n(n-1)\Gamma(\frac{n-1}{2})}{\sqrt{2} \Gamma(\frac{n}{2})}.$$
\end{corollary}

\begin{proof}
  It follows immediately by Corollary \ref{cor.main}.
\end{proof}

\section{Conclusions}
We defined a measure of $\epsilon$ uncontrollability in a Gaussian Random Ensemble of linear systems.  We calculated tight bounds for this probability in terms of $\epsilon$ and the number of states $n$. This is also depicted in the graphs included in the appendix (Figures \ref{fig:2}, \ref{fig:1}).

\appendix
\section{}

  \begin{figure}[H]
    \centering
 \includegraphics{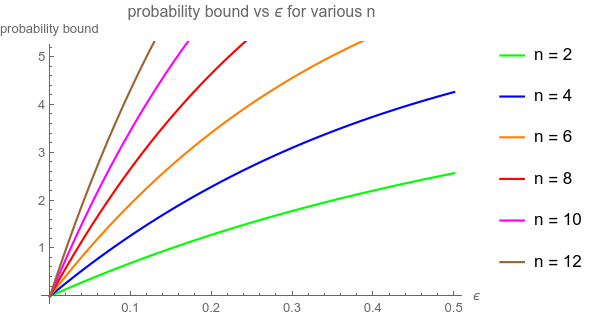}
    \caption{Moving from the lower to the upper value of $\epsilon$ the probability bound increases with n.}
    \label{fig:2}
  \end{figure}

  \begin{figure}[H]
    \centering
    \includegraphics{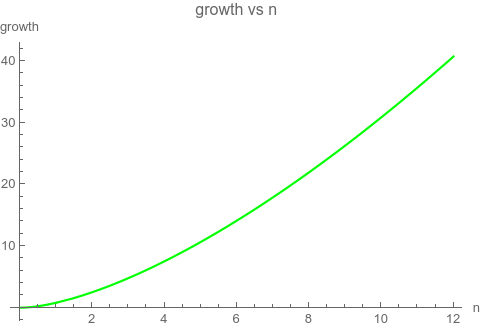}
    \caption{The growth of probability bound increases quadraticaly with n. }
    \label{fig:1}
  \end{figure}
  
  \begin{acknowledgment}
The authors want to express their thanks to professor D. Cheliotis for his valuable suggestions concerning random matrices.
\end{acknowledgment}

\bibliographystyle{amsplain}

\end{document}